\documentclass[11pt,reqno]{amsart}

\usepackage{amsmath,amsfonts,amsthm,amssymb,latexsym,mathrsfs,enumerate}

\usepackage{hyperref}
\usepackage[all]{xy}
\usepackage{tikz-cd}

\newcommand{\diam}{\mathrm{diam}}

\newcommand{\lip}{\mathrm{Lip}^1}

\newcommand{\id}{\mathrm{id}}

\newcommand{\cs}{\mathrm{C}^*}
\newcommand{\aff}{\mathrm{Aff}}

\newcommand{\rr}{\mathbb{R}}

\usepackage{enumitem}

\newenvironment{items}{\begin{list} {$\cdot$} {\setlength{\leftmargin}{0.5cm}}}{\end{list}}

\newtheoremstyle{smallcaps}
    {3pt}                    
    {3pt}                    
    {\itshape}                   
    {}                           
    {\sc}                   
    {.}                          
    {.5em}                       
    {}  
    
\newtheoremstyle{smallcapsdef}
    {3pt}                    
    {3pt}                    
    {}                   
    {}                           
    {\sc}                   
    {.}                          
    {.5em}                       
    {}  

\theoremstyle{plain}

\newtheorem {theorem}{Theorem}[section]
\newtheorem {lemma}[theorem]{Lemma}
\newtheorem {proposition}[theorem]{Proposition}
\newtheorem {corollary}[theorem]{Corollary}
\newtheorem {claim}[theorem]{Claim}

\theoremstyle {definition}

\newtheorem {definition}[theorem]{Definition}
\newtheorem {remark}[theorem]{Remark}

\numberwithin{equation}{section}

\begin{document}

\author{Bhishan Jacelon, Karen R.~Strung, Alessandro Vignati}
\date{\today}
\title{Optimal transport and unitary orbits in $\cs$-algebras}

\address{Department of Mathematics, Physics and Geology\\Cape Breton University\\1250 Grand Lake Road\\Sydney, Nova Scotia B1P 6L2\\Canada} 
\email{bjacelon@gmail.com}
\address{Institute of Mathematics\\Czech Academy of Sciences\\ \v{Z}itn\'{a} 25\\115 67 Praha 1\\Czechia}
\email{strung@math.cas.cz}
\address{Institut de Math\'ematiques de Jussieu - Paris Rive Gauche (IMJ-PRG)\\
Universit\'e de Paris\\
B\^atiment Sophie Germain\\
8 Place Aur\'elie Nemours\\ Paris, 75013, France}
\email{ale.vignati@gmail.com}




\maketitle

\begin{abstract}
Two areas of mathematics which have received substantial attention in recent years are the theory of optimal transport and the Elliott classification programme for $\cs$-algebras.  We combine these two seemingly unrelated disciplines to make progress on a classical problem of Weyl. In particular, we show how results from the Elliott classification programme can be used to translate continuous transport of spectral measures into optimal unitary conjugation in $\cs$-algebras. As a consequence, whenever two normal elements of a sufficiently well-behaved \mbox{$\cs$-algebra} share a spectrum amenable to such continuous transport, and have trivial $K_1$-class, the distance between their unitary orbits can be computed tracially.
\end{abstract}

\section{Introduction}

The problem of calculating the distance between the orbits of operators under conjugation by unitaries was first considered by Weyl \cite{Weyl:1912aa}. It is well known to any student of linear algebra that two normal matrices are in the same unitary orbit if and only if they have the same eigenvalues, including multiplicity. That is to say, a given unitary orbit is completely determined by spectral data.  How then to compare two normal matrices that do not lie in the same orbit? Weyl showed that for self-adjoint matrices, spectral data not only determines a given unitary orbit, but is in fact enough to calculate the distance between two orbits via an optimal matching of eigenvalues. This problem was later considered for unitary matrices \cite{Bhatia:1984aa}, in the setting of von Neumann algebras \cite{Hiai:1989aa, Sherman:2007aa}, purely infinite $\cs$-algebras \cite{Skoufranis:2013aa}, as well as for positive elements in $\cs$-algebras by Toms and the first two authors \cite{Jacelon:2014aa}.   In this article we continue the study of the Weyl problem in the setting of $\cs$-algebras. In particular, we address the question: for which pairs of normal elements in a given well-behaved \mbox{$\cs$-algebra} can one exactly compute the distance between their unitary orbits using only traces? Our approach connects the theory of optimal transport to the theory of $\cs$-algebras via the Elliott classification programme. 

The Elliott classification programme aims to show how two \mbox{$\cs$-algebras} in a given class can be determined, up to isomorphism, by a computable invariant, called the Elliott invariant. Explicitly, the Elliott invariant consists of $K$-theory, traces, and a pairing between these objects.  By now, the classification programme has been proved to be highly successful: all simple separable unital nuclear $\cs$-algebras can be classified by the Elliott invariant under the minor restrictions of $\mathcal{Z}$-stability and provided the so-called universal coefficient theorem (UCT) holds \cite{EllGonLinNiu:ClaFinDecRan, GongLinNiue:ZClass1,  GongLinNiue:ZClass2,  TikWhiWin:QD}. Here, $\mathcal{Z}$-stability refers to tensorial absorption of the Jiang--Su algebra $\mathcal{Z}$. Constructed in \cite{Jiang:1999hb}, $\mathcal{Z}$ plays an essential role in the classification programme as $\mathcal{Z}$-stability guarantees good structural behaviour, see for example the preliminary discussion in \cite{BBSTWW:2Col}.  The UCT \cite{RosSho:UCT} is a relatively mild restriction: although it remains an open question as to whether it holds for all nuclear $\cs$-algebras, it does in fact hold for all \emph{known} nuclear $\cs$-algebras; see \cite{Win:QDQ} for an overview of the role of the UCT in classification.

The aim of this paper, with these powerful results of the Elliott programme in place, is to use the machinery to tackle the classical problem of Weyl.  With a rich selection of models witnessing a given invariant, classification gives us flexibility to choose our favourite representative $\cs$-algebra or $^*$-homomorphism. This was the technique of \cite{Jacelon:2014aa} and is also our strategy in the sequel.

In \cite{Jacelon:2014aa} the Weyl problem for self-adjoint elements was approached globally, by finding a concrete tracial model for the ambient $\cs$-algebra and applying results of the Elliott classification programme.  The Elliott passport allowed us to freely travel from $\cs$-algebras to matrices and settle the question for self-adjoint elements with connected spectra. 

In this paper, our fine-tuning is local: we solve a transportation problem on certain spectra and use our Elliott phrasebook to translate this back to unitary orbits. Optimal transport theory has its beginnings in the 18th century when Monge \cite{Monge} studied how to move a mass from one location to another in a way which is optimal with respect to cost. The problem was later put on solid mathematical grounding and greatly advanced by Kantorovich \cite{Kantorovich:1942aa, Kantorovich2004}, where the problem was rephrased in measure-theoretical terms. Despite having been introduced over two centuries ago, optimal transport continues to find new applications to many areas of mathematics including probability, Riemannian geometry, partial differential equations, and much more, as can be seen for example in \cite{Villani:2009aa, Figalli2011, Villani:2003}.

Here, we study a continuous version of the classical Monge--Kantorovich problem on a compact metric space \cite{Kantorovich:1942aa}, analogous for example to the $L^\infty$ problem considered in \cite{Champion:2008aa}. This concerns the transport of one probability measure onto another, optimal with respect to a given cost function. We show that the quantised version of this problem can be solved for lines, circles, and cubes among other spaces. Since lines and circles are spectra of self-adjoint and unitary elements respectively, this allows us to apply our results to the unitary orbits of such elements. In this context we are primarily interested in compact subsets of the plane, but our analysis is more broadly applicable.  Indeed, alongside the Weyl problem, we consider more generally the minimal distance between unitary conjugates of $^*$-homomorphisms from commutative unital $\cs$-algebras. 

For us, a `well-behaved' $\cs$-algebra will at least be infinite-dimensional, simple, separable, unital, exact,  have stable rank one and strict comparison of positive elements. We also, for the most part, assume real rank zero. (See for example \cite[Chapter \rm{III}]{Blackadar:1998qf} for an explanation of the significance of these assumptions: essentially they ensure that maximal information can be extracted from $K$-theory and traces.) Moving beyond the real rank zero setting would require a more subtle analysis involving determinants of unitaries, but following \cite{Matui:2011uq} we make some progress by assuming $\mathcal{Z}$-stability. Note however that we do not assume nuclearity, so our theorems apply to some algebras not under the umbrella of the Elliott programme (see Remark~\ref{nonnuclear}).

The paper is structured as follows. In \S\ref{measures} we define our transportation property, and show that many well-behaved manifolds satisfy this property.   In \S\ref{distance} we consider Weyl's problem more generally, studying the distance between unitary conjugates of $^*$-homomorphisms from commutative unital $\cs$-algebras. Finally, \S\ref{mainsec} is our bridging of worlds: we show how to use classification to convert transport maps into conjugating unitaries.

\subsection*{Acknowledgements} BJ was partially supported by NSERC of Canada while at the University of Toronto. KRS was partially supported by Sonata 9 NCN grant 2015/17/D/ST1/02529 and a Radboud Excellence Initiative fellowship and is currently funded by currently funded by GA\v{C}R project 20-17488Y and RVO: 67985840.  AV is partially supported by a Prestige co-fund programme and a FWO scholarship. KRS benefited from several visits to the Fields Institute, and the authors would also like to thank l'Institute Math\'ematiques de Jussieu-Paris Rive Gauche for hosting BJ and KRS in June 2018. The authors are  indebted to George Elliott, Cy Maor, Robert McCann and Stuart White for a number of helpful discussions.

\section{Transporting measures} \label{measures}

Let $X$ be a compact connected metric space and endow the space $\mathcal{M}(X)$ of Borel probability measures on $X$ with the usual weak-$^*$ topology. There are many metrisations of $\mathcal{M}(X)$ that appear in the context of optimal transport problems (see \cite{Gibbs:2002aa}), chief among them being the Wasserstein metrics $W_p$ \cite[p.424]{Gibbs:2002aa}. Most relevant for us is a variation of the L\'{e}vy--Prokhorov metric that we will call the \emph{optimal matching distance}:
\begin{equation} 
\delta(\mu,\nu) = \inf\{r\mid\forall U\subseteq X\text{ Borel } (\mu(U) \le \nu(U_r), \: \nu(U) \le \mu(U_r))\},
\end{equation}
where $U_r=\{x\in X \mid d(x,U)<r\}$. The terminology comes from finitely supported measures: if $\mu=\frac{1}{n}\sum_{i=1}^n\delta_{x_i}$ and $\nu=\frac{1}{n}\sum_{i=1}^n\delta_{y_i}$, then (by Hall's marriage theorem)
\begin{equation} \label{delta}
\delta(\mu,\nu) = \min_{\sigma\in S_n} \max_{1\le i \le n} d(x_i, y_{\sigma(i)}).
\end{equation}
In this context the optimal matching distance between $\mu$ and $\nu$ is sometimes also referred to as the \emph{bottleneck distance} between the finite sets $F=\{x_1,\ldots,x_n\}$ and $G=\{y_1,\ldots,y_n\}$, denoted $\mathfrak{b}(F,G)$. Two facts should be remarked:
\begin{itemize}
\item The function $\delta$ is indeed a metric;
\item Although in the literature the bottleneck distance is often defined for multisets (that is, allowing for points to appear multiple times in a list), we only give it for sets. In a nondiscrete metric space every multiset has arbitrarily small bottleneck distance from a set whose elements only appear once.
\end{itemize}

It can be shown that $W_1(\mu,\nu)\le (\diam(X)+1)\delta(\mu,\nu)$ for every pair of measures $\mu,\nu\in\mathcal{M}(X)$ (see \cite[Theorem 2]{Gibbs:2002aa}), but in general $\delta$-convergence is strictly stronger than weak-$^*$ convergence. However, the metric topology induced by $\delta$ coincides with  the weak-$^*$ topology on the dense $G_\delta$ subset of $\mathcal{M}(X)$ consisting of fully supported diffuse measures, denoted $\mathcal{M}_g(X)$. (Using the notation of \cite{Fathi:1980aa}, $g$ stands for `good'.) It is also useful to note that, by tightness of Borel probability measures on compact metric spaces, we may consider only open sets in the definition of $\delta$. In fact, for spaces such as lines and circles it suffices to consider open intervals or arcs; see the proof of \cite[Theorem 2.1]{Hiai:1989aa}.

The following is the $\cs$-analogue of the classical Monge--Kantorovich transportation property.

\begin{definition} \label{maindef}
Let $(X,d)$ be a nonempty compact path-connected metric space. Say that $X$ \emph{admits continuous transport} if it has the following transport property: for every $\mu,\nu\in\mathcal{M}_g(X)$ there exists a homeomorphism $h\colon X\to X$ such that $h_*\nu=\mu$ and
\begin{equation} \label{transport}
\sup_{x\in X}d(h(x),x) \le \delta(\mu,\nu).
\end{equation}
Say that $X$ \emph{admits approximate continuous transport} if it approximately has the transport property; that is, for every $\mu,\nu\in\mathcal{M}_g(X)$ there is a sequence of homeomorphisms $h_n\colon X\to X$ such that $(h_n)_*\nu\to\mu$ in the weak-$^*$-topology and
\begin{equation} \label{approxtransport}
\limsup_n\sup_{x\in X}d(h_n(x),x)\le \delta(\mu,\nu).
\end{equation}
\end{definition}

Recall that a \emph{Peano continuum} is a compact, connected, locally connected metric space, or, equivalently, a continuous image of $[0,1]$. For such a space $X$ and $\mu,\nu\in\mathcal{M}_g(X)$ there exists a homeomorphism $h\colon X\to X$ such that $h_*\nu=\mu$ (see for example \cite[Corollary 9.7.3]{Bogachev:2007aa}). For some applications in the sequel, such as Theorem~\ref{main1}, a continuous surjection will be good enough. That said, if $X$ is a topological manifold then homeomorphisms certainly are on the table. This is the content of the Oxtoby--Ulam Theorem \cite{Oxtoby:1941aa}. At least in this case, then, we are optimising over a nonempty set. One way of looking at the Oxtoby--Ulam Theorem is that, for a fixed $\nu\in\mathcal{M}^\partial_g(X)$, the map
\[
\rho\colon\mathcal{H}^\partial(X)\to M_g^\partial(X), \quad h\mapsto h_*\nu
\]
is surjective. Here, $\mathcal{H}^\partial(X)$ denotes the set of homeomorphisms of $X$ that fix the boundary $\partial X$, and $\mathcal{M}_g^\partial(X)$ consists of those measures in $\mathcal{M}_g(X)$ that are zero on $\partial X$. It was shown in \cite[Theorem 3.3]{Fathi:1980aa} that $\rho$ has a continuous section, at least if one restricts to measures having the same zero sets as $\nu$. A 1-\emph{Lipschitz} section (with respect to the uniform metric and optimal matching distance) would imply continuous transport. The motivating idea of this article is that this is a reasonable demand for spaces that are sufficiently uniform.

Below, we bring this idea to fruition for intervals, arcs and circles. These spaces are of particular interest as they provide spectra (without gaps) of self-adjoint and unitary operators in $\cs$-algebras.  We then generalise this approach to show that every compact convex subset of $\mathbb R^n$ with nonempty interior satisfies the approximate continuous transport property (Theorem~\ref{balls}). The proof of this resembles the one for the circle, but we will treat the circle separately since it presents fewer technical difficulties.

\begin{proposition} \label{transitive}
Suppose that $X\subseteq\mathbb{C}$,  equipped with the Euclidean metric, is either
\begin{enumerate}[label=(\roman*)]
\item an interval, or
\item a circular arc that subtends an angle of at most $\pi$.
\end{enumerate}
Then for any $\nu\in\mathcal{M}_g(X)$, the map
\[
\rho_\nu\colon \mathcal{H}^\partial(X)\to\mathcal{M}_g(X),\quad \rho_\nu(h)=h_*\nu
\]
has a $1$-Lipschitz section $\sigma\colon \mathcal{M}_g(X)\to\mathcal{H}^\partial(X)$. In particular, intervals and circular arcs admit continuous transport.
\end{proposition}

\begin{proof}
(i) For $\mu\in\mathcal{M}_g([0,1])$, define $f_\mu\colon[0,1]\to[0,1]$ by $f_\mu(t)=\mu[0,t)$. Each $f_\mu$ is increasing and continuous, hence an order preserving homeomorphism. Fix $\nu\in\mathcal M_f([0,1])$, and define $\sigma\colon \mathcal{M}_g([0,1])\to\mathcal{H}^\partial([0,1])$ by $\sigma(\mu)=f_\mu^{-1}\circ f_\nu$. In other words, $\sigma(\mu)$ is the well-known `increasing rearrangement' map; see \cite[Chapter 1]{Villani:2009aa}. First we check that $\sigma$ is right-inverse to $\rho_\nu$: for $\mu\in\mathcal{M}_g([0,1])$ and $t\in[0,1]$ we have
\[
\rho_\nu\sigma(\mu)[0,t) = \nu(\sigma(\mu)^{-1}[0,t)) = \nu[0,f_\nu^{-1}f_\mu(t)) = f_\mu(t) = \mu[0,t)
\]
which implies that $\rho_\nu\sigma(\mu)=\mu$.

It remains to check that $\sigma$ is $1$-Lipschitz (with respect to the uniform metric and optimal matching distance). Let $\mu_1,\mu_2\in\mathcal{M}_g([0,1])$ and write $\delta(\mu_1,\mu_2) = r$. For $t\in[0,1]$ and $i\in\{1,2\}$, let $s_i\in[0,1]$ such that $f_{\mu_i}(s_i) = f_\nu(t)$. That is, $s_i=\sigma(\mu_i)(t)$. Since
\[
f_{\mu_2}(s_1-r) \le f_{\mu_1}(s_1) = f_{\mu_2}(s_2) \le f_{\mu_2}(s_1+r),
\]
it follows that $s_1-r\le s_2\le s_1+r$. Thus
\[
\sup_{t\in[0,1]} |\sigma(\mu_1)(t) - \sigma(\mu_2)(t)| \le \delta(\mu_1,\mu_2).
\]
To deduce continuous transport, take $\mu_1=\mu$, $\mu_2=\nu$.

(ii) For some $\theta\in (0, \pi]$, the map $\gamma \colon [0,1] \to e^{i t \theta}$ defines a homeomorphism and for any $\mu \in \mathcal{M}_g(X)$ we have $\mu \circ \gamma \in \mathcal{M}_g([0,1])$. Fix $\nu\in\mathcal M_g(X)$, and define 
\[ \sigma \colon \mathcal M_g(X) \to \mathcal{H}^{\delta}(X)\]
by 
\[ \sigma(\mu)(e^{it\theta}) = f_{\mu\circ\gamma}^{-1} \circ f_{\nu \circ \gamma}(t),\]
where, for $\mu \in \mathcal{M}_g([0,1])$, the function $f_{\mu}$ is as in (i). Proceeding as above, it is easy to check that $\sigma$ is a right-inverse to $\rho_\nu$.

Now suppose that $\delta(\mu_1, \mu_2) = d$. Then since $\theta \in (0, \pi]$, we have that $\delta_{[0,1]}( \mu_1 \circ \gamma, \mu_2 \circ \gamma) = \frac{2}{\theta} \sin^{-1}(d/2)$, since $\delta_{[0,1]}$ is defined by the arc-length metric $\rho_1$, normalised by the length of the arc $X$. By part (i), we have
\[ \sup_{t\in [0,1]} \theta^{-1} \rho_1(\sigma(\mu_1 \circ \gamma)(t), \sigma(\mu_2 \circ \gamma)(t)) \leq \frac{2}{\theta} \sin^{-1}(d/2).\]
Thus for any $t \in [0,1]$
\[ \frac{1}{2} | \sigma(\mu_1 \circ \gamma)(t) - \sigma(\mu_2 \circ \gamma)(t) | \leq d/2,\]
from which it follows that
\[ \sup_{x \in X} | \sigma(\mu_1)(x) - \sigma(\mu_2)(x)| \leq \delta(\mu_1, \mu_2).\]
Hence $\sigma$ is $1$-Lipschitz, as required.
\end{proof}

The following provides a means of demonstrating approximate continuous transport both for the circle $\mathbb{T}$  and for compact convex subsets of $\mathbb{R}^n$ with nonempty interior.

\begin{lemma} \label{finiteapproximation}
Let $(X,d)$ be a compact metric space, let $\mu\in\mathcal{M}_g(X)$ and let $\varepsilon>0$. Then there exists $\mu'\in\mathcal{M}(X)$ of finite support such that $\delta(\mu,\mu')<\varepsilon$.
\end{lemma}

\begin{proof}
It is not hard to show that there exist pairwise disjoint open subsets $U_1,\ldots,U_m$ of $X$, each of diameter $<\varepsilon$, such that $\mu(X\setminus\bigcup_{i=1}^m U_i) = 0$. (See for example the proof of \cite[Lemma 4.1]{Matui:2011uq}.) We may also assume that there are natural numbers $M$ and $k_1,\ldots,k_m$ such that $\sum k_i=M$ and $\mu(U_i)=\frac{k_i}{M}$ for every $i$. For each $i$, choose distinct points $x_1^i,\ldots,x_{k_i}^i\in U_i$. Set $\mu'=\frac{1}{M}\sum_{i,j}\delta_{x_j^i}$.  

Now fix an open set $U\subseteq X$. Let $p$ be minimal so that (after reordering) $\mu(U\setminus\bigcup_{i=1}^pU_i) =0$. By minimality, $U\subseteq (\bigcup_{i=1}^p U_i) \cup E$ where $E \subset X\setminus \bigcup_{i=1}^m U_i$ is a set of measure zero.  In particular $\delta_{x^i_j}(E) = 0$ for any $i > p$, $1\leq j \leq k_i$. Thus 
\begin{align*}
\mu'(U) &\le \mu'\left(\bigcup_{i=1}^pU_i\right) = \mu\left(\bigcup_{i=1}^pU_i\right) \le\mu(U_\varepsilon),
\end{align*}
and the other way round. It follows that $\delta(\mu,\mu')<\varepsilon$.
\end{proof}

\begin{lemma} \label{ordering}
Let $n >0$ and $T, S \subseteq \mathbb{T}$ with $|T|=|S|=n$. Define Borel probability measures $\mu = \frac{1}{n}\sum_{t\in T} \delta_{t}$ and $\nu = \frac{1}{n}\sum_{s \in S} \delta_{s}$. Then if $s_1, \dots, s_n$ is an anticlockwise ordering of $S$ there is an anticlockwise ordering $t_1,\ldots,t_n$ of $T$ such that \[ \max_{1\le i\le n}|s_i - t_i| = \delta(\mu, \nu).\]
\end{lemma}

\begin{proof}
This follows from the proof of the main result of \cite{Bhatia:1984aa}. See also Remark~\ref{matching} below.
\end{proof}

\begin{proposition} \label{circles}
Circles admit approximate continuous transport.
\end{proposition}

\begin{proof}
Let $\mu,\nu\in\mathcal{M}_g(\mathbb{T})$. By Lemma~\ref{finiteapproximation}, for every $n\in\mathbb{N}$ there exist $M_n\in\mathbb{N}$ and finite sets $T_n'=\{s_1^n,\ldots,s_{M_n}^n\}$ and $T_n=\{t_1^n,\ldots,t_{M_n}^n\}$ such that  $\mu_n = \frac{1}{M_n}\sum_{t\in T_n} \delta_{t}$ and $\nu_n = \frac{1}{M_n}\sum_{s \in T_n'}\delta_{s}$ are probability measures with
\[
\max\{\delta(\mu_n,\mu),\delta(\nu_n,\nu)\} < \frac{1}{2^n}.
\]
By Lemma~\ref{ordering} we may assume $T_n'=\{s_1^n,\ldots,s_{M_n}^n\}$ and $T_n=\{t_1^n,\ldots,t_{M_n}^n\}$ are ordered anticlockwise so that  \[ \max_{1\le i\le M_n}|s_i^n - t_i^n| = \delta(\mu_n,\nu_n).\] Define $h_n\colon \mathbb{T}\to\mathbb{T} $ by setting $h_n(s_i^n)=t_i^n$ and interpolating logarithmically linearly (so each arc $[s_i^n,s_{i+1}^n]$ is mapped onto $[t_i^n,t_{i+1}^n]$). Then each $h_n$ is a homeomorphism such that $(h_n)_*\nu_n=\mu_n$ and $\sup_{t\in\mathbb{T}}|h_n(t)-t| = \delta(\mu_n,\nu_n)$. Notice that each $h_n$ is $\delta(\mu_n,\nu_n)$-Lipschitz, and therefore, for $n$ large enough, each $h_n$ is $2\delta(\mu,\nu)$-Lipschitz. Then we have
\begin{align} \label{circleargument}
\delta((h_n)_*\nu,\mu) &\le \delta((h_n)_*\nu,(h_n)_*\nu_n) + \delta((h_n)_*\nu_n,\mu_n) + \delta(\mu_n,\mu) \\
&\le 2\delta(\mu,\nu)\cdot\delta(\nu,\nu_n) + \delta(\mu_n,\mu) \nonumber \\
& \to 0 \quad \textrm {as } n\to\infty, \nonumber,
\end{align}
so $(h_n)_*\nu\to^{w^*}\mu$.
\end{proof}

\begin{remark} \label{polygons}
By a similar argument, one can show that there is a decreasing sequence $(c_k)_{k=1}^\infty\subseteq\mathbb{R}$ with $c_k\to1$ as $k\to\infty$ such that the boundary of a regular $k$-gon $X_k$ admits approximate continuous transport scaled by $c_k$. That is, for any $\mu,\nu\in\mathcal{M}_g(X_k)$, (\ref{approxtransport}) holds with upper bound $c_k\delta(\mu,\nu)$. In general one would expect the best constant $c_X$ for a space $X$ to depend on its geometry: for example, if $X$ is formed by two copies of $[0,1]$ attached at $0$, then $c_X$ depends on the angle these two intervals form, and $c_X\to1$ as the angle tends to $\pi$.
\end{remark}

Our next goal is to prove that, for $n \geq 2$, compact convex subsets of $\mathbb R^n$ with nonempty interior have approximate continuous transport.  The strategy is similar to the one used in Proposition~\ref{circles}, hence we need to optimally match large finite sets in $\rr^n$. This is quite easy if $n\geq 3$, and becomes trickier when $n=2$.

\begin{definition}
Let $F$ and $G$ be finite subsets of $\rr^n$ of the same size. We say that $F=\{x_1,\ldots,x_m\}$ and $G=\{y_1,\ldots,y_m\}$ are \emph{enumerated optimally} if for all $i\leq m$,
\begin{equation} \label{bottleneck}
d(x_i,y_i)=\mathfrak b(\{x_1,\ldots,x_i\},\{y_1,\ldots,y_i\}),
\end{equation}
where $\mathfrak b$ is the bottleneck distance defined after equation~\eqref{delta}.
\end{definition}
Given finite sets of the same size, say $F$ and $G$, it is always possible to enumerate them optimally. For this, enumerate $F=\{f_1,\ldots,f_m\}$ and $G=\{g_1,\ldots,g_m\}$, and let $\sigma\in S_m$ be a bijection witnessing that $\mathfrak b(F,G)=\max_i d(f_i,g_{\sigma(i)})$. Let $i$ be the minimal index such that $d(f_i,g_{\sigma(i)})$ is maximal, and set $x_m=f_i$ and $y_m=g_{\sigma(i)}$. We then repeat such process on $F\setminus\{x_m\}$ and $G\setminus\{y_m\}$. 
If $X$ is a compact convex subset of $\rr^n$, we denote the interior of $X$ by $\mathring X$. 

\begin{proposition} \label{prop:opensets}
Let $n\geq 3$, and let $X\subseteq\rr^n$ be a compact convex subset with $\mathring X\neq\emptyset$. Let $F=\{x_1,\ldots,x_m\}$ and $G=\{y_1,\ldots,y_m\}$ be two finite disjoint subset of $\mathring{X}$ which are enumerated optimally. Then there are mutually disjoint paths $\alpha_i\colon [0,1]\to\mathring{X}$, $i\in\{1,\ldots,m\}$, with $\alpha_i(0)=x_i$, $\alpha_i(1)=y_i$ and
\begin{equation} \label{mid}
\alpha_i(t)\subseteq \left \{z \,\mid d(z,\frac{x_i+y_i}{2})<\frac{1}{2}\mathfrak b(F,G) \right \} \quad \textrm{ for all } t\in (0,1).
\end{equation}
In particular, for every $i\in\{1,\ldots,m\}$,
\begin{equation} \label{shortpaths}
\sup_{y,z\in \alpha_i}d(y,z) \le \mathfrak{b}(F,G).
\end{equation}
\end{proposition}

\begin{proof}
Connect each $x_i$ to $y_i$ by a straight line $\tilde{\alpha}_i$. Wherever $\tilde{\alpha}_1$ crosses another path, perturb it by a small circular arc in a plane containing $\tilde{\alpha}_1$ but no other $\tilde{\alpha}_i$, and set this path as $\alpha_1$. Continue to replace each $\tilde{\alpha}_i$ with a possibly perturbed path $\alpha_i$ in turn. Note that straight lines perturbed in this way will satisfy (\ref{mid}): since $F$ and $G$ are disjoint the points of intersections of $\tilde\alpha_i$ and $\tilde\alpha_j$ are in the interior of $\{z\mid d(z,(x_i+y_i)/2),d(z,(x_j+y_j)/2)\leq\mathfrak b(F,G)\}$. Moreover, as long as we did not perturb each $\tilde{\alpha}_i$ too much, we can ensure that our paths are completely contained in $\mathring{X}$.
\end{proof}

When $n = 2$ there is less room to manoeuvre. We identify a path $\alpha$ with its image $\alpha([0,1])$. 

If $X\subseteq\rr^2$ is a compact convex set with nonempty interior, and $F=\{x_1,\ldots,x_m\}$ and $G=\{y_1,\ldots,y_m\}$ are disjoint subsets of $\mathring{X}$, we denote by $\mathcal P(X,F,G)$ the class of all $m$-ples $\bar \alpha=(\alpha_1,\ldots,\alpha_m)$ where, for each $i\leq m$, \begin{itemize}
\item  $\alpha_i$ is an injective path contained in $\mathring{X}$ connecting $x_i$ to $y_i$, 
\item if $i\neq j$ then $\alpha_i\cap\alpha_j$ is finite,
\item if $i\neq j$ then $x_i\notin\alpha_j$ and $y_i\notin \alpha_j$, and
\item  $\alpha_i((0,1))\subseteq B_{\varepsilon_i/2}(\frac{x_i+y_i}{2})$, where $\varepsilon_i=d(x_i,y_i)$.
\end{itemize} 
Notice that $\mathcal P(X,F,G)$ is nonempty: the only nontrivial condition, that is, that $\alpha_i\cap\alpha_j$ is finite whenever $i\neq j$, can be ensured by the fact that piecewise linear functions are dense in the set of paths. The third condition can be ensured as the two sets are disjoint.

\begin{proposition} \label{prop:opensets2}
Let $X\subseteq\rr^2$ be a convex subset with $\mathring X\neq\emptyset$. Let $F=\{x_1,\ldots,x_m\}$ and $G=\{y_1,\ldots,y_m\}$ be finite disjoint subsets of $\mathring{X}$ which are optimally enumerated. Then there is $\bar\alpha\in\mathcal P(X,F,G)$ with the property that $\alpha_i\cap\alpha_j=\emptyset$ whenever $i\neq j$.
\end{proposition}

\begin{proof} 
The whole argument takes place in $\mathring{X}$, so we can assume $X$ is open.
We will work by contradiction, and suppose that $F$, $G$ and $X$ are as in the hypothesis and such that whenever $\bar\alpha\in \mathcal P(X,F,G)$ then there are $i$ and $j$ with $i\neq j$ and $\alpha_i\cap\alpha_j\neq\emptyset$. Let $z_i=\frac{x_i+y_i}{2}$, $\varepsilon_i=d(x_i,y_i)$, and let $\mathcal P=\mathcal P(X,F,G)$. Since each $x_i$ and each $y_i$ are in $X$, so does $z_i$. 

For $\bar\alpha\in\mathcal P$, define
\[
i_{\bar\alpha}=\max \{i\mid\exists j<i(\alpha_i\cap\alpha_j\neq\emptyset)\}, \text{ and } j_{\bar\alpha}=\max\{j< i_{\bar\alpha}\mid \alpha_j\cap\alpha_{i_{\bar\alpha}}\neq\emptyset\}.
\]
Pick $\bar\alpha\in\mathcal P$ minimal for the choice of $i_{\bar\alpha}$, $j_{\bar\alpha}$ and $|\alpha_{j_{\bar\alpha}}\cap\alpha_{i_{\bar\alpha}}|$, in the sense that whenever $\bar\gamma\in\mathcal P$ then
\begin{itemize}
\item $i_{\bar\gamma}\geq i_{\bar\alpha}$
\item if $i_{\bar\gamma}=i_{\bar\alpha}$ then $j_{\bar\gamma}\geq j_{\bar\alpha}$, and
\item if $i_{\bar\gamma}= i_{\bar\alpha}$ and $j_{\bar\gamma}= j_{\bar\alpha}$, then
\[
|\alpha_{j_{\bar\alpha}}\cap\alpha_{i_{\bar\alpha}}|\leq |\gamma_{j_{\bar\gamma}}\cap\gamma_{i_{\bar\gamma}}|.
\]
\end{itemize}
Set 
\[
\bar i=i_{\bar\alpha}, \, \bar j=j_{\bar\alpha}, \, \alpha=\alpha_{\bar i}, \,\text{ and } \beta=\alpha_{\bar j}.
\]
We will modify $\alpha$ and $\beta$ to injective paths $\alpha'$ and $\beta'$ with the same endpoints and the property that 
\[
|\alpha'\cap\beta'|<|\alpha\cap\beta|, \, \alpha'((0,1))\subseteq B_{\varepsilon_{\bar i}}(z_{\bar i}), \,\beta'((0,1))\subseteq B_{\varepsilon_{\bar j}}(z_{\bar j}),
\]
and moreover
\begin{itemize}
\item if $i>\bar i$, then $\alpha'\cup\beta'$ and $\alpha_i$ are disjoint,
\item and if $j$ is such that $\bar j<j<\bar i$ then $\alpha'\cap\alpha_j=\emptyset$.
\end{itemize} 
Setting $\bar\alpha'\in\mathcal P$ by $\alpha'_i=\alpha_i$ if $i\neq\bar i,\bar j$, $\alpha'_{\bar i}=\alpha'$ and $\alpha'_{\bar j}=\beta'$ would contradict the minimality of $|\alpha_{j_{\bar\alpha}}\cap\alpha_{i_{\bar\alpha}}|$. Therefore, once $\alpha'$ and $\beta'$ will be constructed as above, the proposition would be proved.

Let $B=B_{\varepsilon_{\bar i}/2}(z_{\bar i})$ and let $C=\partial B$ be its boundary. The set $C\setminus\{x_{\bar i},y_{\bar i}\}$ has two connected components $C_n$ and $C_s$ (here the subscript $n$ and $s$ are for `north' and `south'). Similarly, since $\alpha_{\bar i}$ is injective, $B\setminus\alpha_{\bar i}$ has two connected components $B_n$ and $B_s$. We choose the indexes so that 
\[
\partial B_n=C_n\cup\alpha\text{ and }\partial B_s=C_s\cup\alpha.
\]
\begin{claim}
Not both $x_{\bar j}$ and $y_{\bar j}$ are in $\overline B$. Therefore $z_{\bar j}\neq z_{\bar i}$, and in particular the function $f\colon C\to\rr$ given by $f(w)=d(w,z_{\bar j})$ is not constant.
\end{claim}
\begin{proof}
Suppose both $x_{\bar j}$ and $y_{\bar j}$ are in $\overline B$. Since $F$ and $G$ are disjoint, both $x_{\bar j}$ and $y_{\bar j}$ are in $\overline B\setminus\{x_{\bar i},y_{\bar i}\}$. Since $y_{\bar i}$ is the farthest point of $B$ fom $x_{\bar i}$, then $d(x_{\bar i},y_{\bar j})<\varepsilon_{\bar i}$. Similarly $d(y_{\bar i},x_{\bar j})<\varepsilon_{\bar i}$. Since 
\[
\varepsilon_{\bar i}=d(x_{\bar i},y_{\bar i})=\mathfrak b(\{x_i\}_{i\leq\bar i},\{y_i\}_{i\leq\bar i}),
\] this is a contradiction to the fact that $F$ and $G$ are optimally enumerated.

For the second statement, notice that $\max\{d(z_{\bar j},x_{\bar j}),d(z_{\bar j},y_{\bar j})\}\leq\varepsilon_{\bar i}/2$. Since $B$ is the ball of radius $\varepsilon_{\bar i}/2$ around $z_{\bar i}$, if $z_{\bar i}=z_{\bar j}$ then $x_{\bar j},y_{\bar j}\in\overline B$, a contradiction.
\end{proof}
 Let $\delta>0$ such that $d(\alpha\cup\beta,\bigcup_{i>\bar i}\alpha_i)>\delta$. Notice that $x_{\bar i}$ and $y_{\bar i}$ are not in $\beta$ and $x_{\bar j}$ and $y_{\bar j}$ are not in $\alpha$. Since $\alpha\cap\beta\neq\emptyset$ there are $t$ and $u$ in $(0,1)$ such that $\alpha(u)=\beta(t)$. Since $\alpha\cap\beta$ is finite, there is $\eta>0$ such that $\beta((t-\eta,t+\eta))\cap\alpha=\beta(t)$. We can moreover ask that $\eta$ is small enough so that 
\[
\alpha([u-\eta,u+\eta])\cup\beta([t-\eta,t+\eta])\subseteq B_\delta(\alpha(u))\text{ and }\alpha([u-\eta,u+\eta])\cap\beta=\alpha(u).
\]
 Recall that $\alpha$ divides $B$ into two connected components $B_n$ and $B_s$. Suppose that $\beta((t-\eta,t))$ and $\beta((t,t+\eta))$ are contained in the same connected component, say $B_n$ (this happens if $\alpha$ and $\beta$ are tangent to each other). Then $\beta(t-\eta)$ and $\beta(t+\eta)$ are in the same path connected component of $B_\delta(\alpha(u))\cap B_n\cap B_{\varepsilon_j/2}(z_{\varepsilon_j})$. Let $\gamma$ be a path connecting $\beta(t-\eta)$ to $\beta(t+\eta)$ completely contained in $B_\delta(\alpha(u))\cap B_n\cap B_{\varepsilon_j/2}(z_{\varepsilon_j})$, and let
 \[
 \beta'=\beta([0,t-\eta])^{\frown}\gamma^{\frown}\beta([t+\eta,1]).
\]
Then $|\alpha\cap\beta'|=|\alpha\cap\beta|-1$, and this, setting $\alpha'=\alpha$, would lead to a contradiction. Modulo inverting north and south, we can then assume that 
 \[
 \beta((t-\eta,t))\subseteq B_n\text{ and }\beta((t,t+\eta))\subseteq B_s.
 \]
Let $r_1=\min\{r>t+\eta\mid \beta(r)\notin B_s\}$. We have three cases:

\begin{enumerate}[label=\underline{C \arabic*}]
\item $\beta((t,1])\subseteq B_s$, that is, $r_1$ cannot be defined.

\item $\beta((t,1])$ leaves $B_s$ through $\alpha$, that is, $\beta(r_1)\in\alpha$.

\item  $\beta((t,1])$ leaves $B_s$ through $C_s$, that is, $\beta(r_1)\in C_s$
\end{enumerate}
We will now show that both Case 1 and 2 lead to a contradiction, and hence that we have Case 3. (A similar argument shows that $\beta([0,t))$ enters $B_n$ from $C_n$).

\noindent\underline{CASE 1}: $\beta((t,1])\subseteq B_s$. Let 
\[
D_1=(B_{\delta}(\beta((t,1]))\cap B_s\setminus(\alpha\cup\beta))\cup\{\alpha(u-\eta),\alpha(u+\eta)\}.
\] 
Notice that $\alpha(u-\eta)$ and $\alpha(u+\eta)$ are in the same path connected component of $D_1$. Hence there is an injective path $\gamma\colon [u-\eta,u+\eta]\to D_1$ such that 
\[
\gamma(u-\eta)=\alpha(u-\eta)\text{ and }\gamma(u+\eta)=\alpha(u+\eta).
\]
 Let 
\[
\alpha'=\alpha([0,u-\eta])^{\frown}\gamma^{\frown} \alpha([u+\eta,1]).
\]
 Notice that $\alpha'\subseteq B_{\varepsilon_{\bar i}}(z_{\bar i})$, $\alpha'\cap\alpha_i=\emptyset$ if $i>\bar i$, and $|\alpha'\cap\beta|=|\alpha\cap\beta|-1$. This is a contradiction to the minimality of $|\alpha\cap\beta|$ over all elements of $\mathcal P$ such that $i_{\bar\alpha}=\bar i$ and $j_{\bar\alpha}=\bar j$.

\noindent\underline{CASE 2}: $\beta((t,1])$ leaves $B_s$ through $\alpha$. Pick $t_1=\min \{r>t\mid \beta(r)\in \alpha\}$, and $u_1$ such that $\alpha(u_1)=\beta(t_1)$. We suppose that $u_1>u$. Pick $\eta'<\eta$ small enough so that 
\[
\alpha((u_1-\eta',u_1+\eta'))\cap\beta((t_1-\eta',t_1+\eta'))=\alpha(u_1)
\]
and
\[
\alpha((u_1-\eta',u_1+\eta'))\subseteq B_{\delta}(\alpha(u_1)).
\]
 Let
\[
D_2=B_\delta(\beta(t,t_1))\cap B_s\setminus (\alpha\cup\beta))\cup\{\alpha(u-\eta),\alpha(u_1+\eta')\}.
\]
Notice that $\alpha(u-\eta)$ and $\alpha(u_1+\eta')$ are in the same path connected component of $D_2$, and pick $\gamma\colon [u-\eta,u_1+\eta']\to D_2$ be a path with $\gamma(u-\eta)=\alpha(u-\eta)$ and $\gamma(u_1+\eta')=\alpha(u_1+\eta')$. Then the path $\alpha'$ given by 
\[
\alpha'=\alpha([0,u-\eta])^{\frown}\gamma^{\frown}\alpha([u_1+\eta',1])
\]
 is an injective path from $x_{\bar i}$ to $y_{\bar i}$ such that $\alpha'((0,1))\subseteq B$, which is disjoint from each $\alpha_i$ from $i>\bar i$, and that has the property that $|\alpha'\cap\beta|\leq|\alpha\cap\beta|-1$, a contradiction. The case $u_1<u$ is treated in the same exact way.

Since both Case 1 and 2 lead to a contradiction, we are in the setting of Case 3. In particular, we showed that $\forall t\in(0,1)$, if $\beta(t)\in\alpha$, then there are $r_1$ and $r_2$ with $r_1<t<r_2$ such that $\beta((r_1,r_2))\subseteq B$ and $\beta(r_1)$ and $\beta(r_2)$ are in two different connected components of $C\setminus\{x_{\bar i},y_{\bar i}\}$.

Recall that $f\colon C\to\mathbb R$ was defined as $f(w)=d(w,z_{\bar j})$. Since the intersection of two different circles has size at most $2$, $|f^{-1}(z)|\leq 2$ for all $z\in\rr$.
\begin{claim}\label{claim1}
Let $r_1$ and $r_2$ be different points in $[0,1]$, and suppose that $\beta(r_1)$ and $\beta(r_2)$ are in $C$. Then one of the two connected components of $C\setminus\{\beta(r_1),\beta(r_2)\}$ is completely contained in $B_{\varepsilon_{\bar j}/2}(z_{\bar j})$.
\end{claim}
\begin{proof}
Since $\beta\subseteq \overline {B_{\varepsilon_{\bar j}/2}(z_{\bar j})}$, then 
\[
f\restriction\beta\cap C\subseteq[0,\varepsilon_{\bar j}/2].
\]
 Let $C_1$ and $C_2$ be the two connected components of $C\setminus\{\beta(r_1),\beta(r_2)\}$, and suppose that there are $w_1\in C_1$ and $w_2\in C_2$ with $z'\geq\varepsilon_{\bar j}/2$ where $z'=\min\{f(w_1),f(w_2)\}$. Let $z=\max\{f(\beta(r_1)),f(\beta(r_2))\}$. Notice that by its definition $f$ has one minimum and one maximum. (The two points of $C$ giving a minimum and a maximum to $f'$ are obtained by looking at the intersection of the line connecting $z_{\bar i}$ to $z_{\bar j}$ with $C$). Therefore $f'$, the derivative of $f$ has only two zeros. But, if $z\leq \varepsilon_{\bar j}/2\leq z'$, then $f'$ has a zero in each of the four connected components of $C\setminus \{w_1,w_2,\beta(r_1),\beta(r_2)\}$, a contradiction.
\end{proof}
Since $|\alpha\cap\beta|\geq 1$, there are $r_1$ and $r_2$ such that $\beta(r_1)$ and $\beta(r_2)$ are in $C$. Since $x_{\bar i}$ and $y_{\bar i}$ are not in $\beta$, then either $x_{\bar i}$ or $y_{\bar i}$ is in the connected component of $C\setminus\{\beta(r_1),\beta(r_2)\}$ whose $f$ value is $<\varepsilon_{\bar j}/2$, by Claim~\ref{claim1}. Without loss of generality, assume that $f(x_{\bar i})<\varepsilon_{\bar j}/2$. 

Let $u=\min\{t\mid \alpha(t)\in\beta\}$, and $t'$ such that $\alpha(u)=\beta(t')$. Let $r_1<t'<r_2$ be such that $\beta(r_1)$ and $\beta(r_2)$ are in two different connected components of $C\setminus\{x_{\bar i},y_{\bar i}\}$ and $\beta((r_1,r_2))\subseteq B$. Let $C_1$ be the connected component of $x_{\bar i}$ in $C\setminus\{\beta(r_1),\beta(r_2)\}$. Since $f(y_{\bar i})>\varepsilon_{\bar j}/2$, $f(C_1)<\varepsilon_{\bar j}/2$, hence $C_1\subseteq B_{\varepsilon_{\bar j}/2}(z_{\bar j})$. Therefore $\alpha([0,u])$ is included in the interior of a connected set delimited by $C_1$ and $\beta([r_1,r_2])$ which is contained in $\overline{B_{\varepsilon_{\bar j}/2}(z_{\bar j})}$, and therefore $\alpha([0,u])\subseteq B_{\varepsilon_{\bar j}/2}(z_{\bar j})$. We are now going to modify $\beta$. Let $\eta>0$ such that $\beta([t'-\eta,t'+\eta])\subseteq B_\delta(\alpha(u))$ and $\beta([t'-\eta,t'+\eta])\cap\alpha([u-\eta,u+\eta])=\alpha(u)$.
Let
\[
D_3=(B_\delta(\alpha([0,u])\cap B_{\varepsilon_{\bar j}/2}(z_{\bar j})\setminus(\alpha\cup\beta))\cup\{\beta(t'-\eta),\beta(t'+\eta)\}.
\]
Then $\beta(t'-\eta)$ and $\beta(t'+\eta)$ are in the same connected component of $D_3$. Let $\gamma\colon[t'-\eta,t'+\eta]\to D_3$ be an injective path such that $\gamma(t'-\eta)=\beta(t'-\eta)$ and $\gamma(t'+\eta)=\beta(t'+\eta)$. Then $\alpha'_{\bar j}=\beta([0,t'-\eta])^{\frown}\gamma^{\frown}\beta([t'+\eta,1])$ is an injective path between $x_{\bar j}$ and $y_{\bar j}$ such that $|\alpha'_{\bar j}\cap\alpha_{\bar i}|\leq |\alpha_{\bar j}\cap\alpha_{\bar i}|-1$, $\alpha'_{\bar j}((0,1))\subseteq B_{\varepsilon_{\bar j}/2}(z_{\bar j})$ and if $i>\bar i$, $\alpha'_{\bar j}\cap\alpha_i=\emptyset$. This is a contradiction.
\end{proof}

\begin{lemma}\label{lemma:minimaldistance}
Let $n \geq 2$ and $m\in\mathbb N$. Let $X$ be a compact convex subset of $\mathbb R^n$ with $\mathring{X}\neq\emptyset$. Let $F=\{x_1,\ldots,x_m\}$ and $G=\{y_1,\ldots,y_m\}$ be finite disjoint subsets of $\mathring{X}$ which are optimally enumerated, and let $\varepsilon>0$. Then there is a homeomorphism $h\colon X\to X$ such that $h(F)=G$ and
\[
\sup_{x\in X}d(h(x),x)<\mathfrak b(F,G)+\varepsilon.
\]
\end{lemma}

\begin{proof}
Let $\alpha_1,\ldots,\alpha_m$ be the paths established in Proposition~\ref{prop:opensets} if $n\geq 3$ or  Proposition~\ref{prop:opensets2} if $n =2$. We can find a positive number $\delta$ small enough such that $\delta<\min\{\varepsilon,\frac{1}{3}\min_{i\neq j}d(\alpha_i,\alpha_j)\}$ and, for each $i$, $U_i=\bigcup_{z\in \alpha_i}B_{\delta}(z)$ is homeomorphic to $(0,1)^n$. We can then find a homeomorphism $h_i\colon \overline{U_i}\to\overline{U_i}$ such that $h_i(x_i)=y_i$ and $h_i(z)=z$ if $z\in\partial {U_i}$. Since the $U_i$'s are disjoint,  $h=\bigcup_{1\le i\le m}h_i$ is the required homeomorphism.
\end{proof}

\begin{theorem} \label{balls}
Let $X\subseteq\rr^n$ be a compact convex subset with $\mathring X\neq\emptyset$. Then $X$ admits approximate continuous transport.
\end{theorem}

\begin{proof}
This now follows from Lemma~\ref{finiteapproximation}, Lemma~\ref{lemma:minimaldistance} and the triangle inequality (see \ref{circleargument}).
\end{proof}

\section{Distances between unitary orbits} \label{distance}

In what follows, let $A$ be a simple, tracial, unital $\cs$-algebra and let $X$ be a compact metric space. Denote the unitary group of $A$ by $\mathcal{U}(A)$ and the space of tracial states on $A$ by $T(A)$. Write $\lip(X)$ for the set of 1-Lipschitz functions $X\to[0,\infty)$. If $\tau\in T(C(X))$, denote by $\mu_\tau$ the Borel probability measure on $X$ corresponding under Riesz representation to $\tau$. (That is, $\tau(f) = \int f d\mu_\tau$ for every $f\in C(X)$.) Dually, for a Borel probability measure $\mu$ on $X$, denote by $\tau_\mu$ the trace on $C(X)$ defined by $\tau_\mu(f)=\int f d\mu$.

For two positive elements $a,b\in A$ we say that $a$ is Cuntz subequivalent, or Cuntz below, $b$, and write $a\precsim b$, if there are elements $x_n\in A\otimes\mathcal K$ with $x_n^*bx_n\to a$. The Cuntz relation measures, in some sense, the inclusion of spectral projections (see \cite[\S2]{AraPereraToms.Class}). Recall that a simple, unital, exact $\cs$-algebra $A$ is said to have \emph{strict comparison of positive elements}, or simply \emph{strict comparison}, if $a\precsim b$ whenever $a,b\in A$ are positive elements such that $\lim_{n\to\infty}\tau(a^{\frac{1}{n}}) < \lim_{n\to\infty}\tau(b^{\frac{1}{n}})$ for every $\tau\in T(A)$. (For short, we will wirte $d_\tau(a)$ for $\lim_{n\to\infty}\tau(a^{\frac{1}{n}})$ in what follows.) Strict comparison holds whenever $A$ is $\mathcal{Z}$-stable (see \cite[Corollary 4.6]{Rordam:2004kq}) or, assuming real rank zero and stable rank one, if $K_0(A)$ is weakly unperforated (see \cite[Corollary 6.9.2]{Blackadar:1998qf} and \cite[Corollary 3.10]{Perera:1997zl}).

We introduce distances on the set of equivalence classes, under approximate unitary equivalence, of unital \mbox{$^*$-homomorphisms} $C(X)\to A$.

\begin{definition} \label{metrics}
If $\varphi,\psi\colon C(X)\to A$ are unital $^*$-monomorphisms, define the \emph{unitary distance} as
\begin{align}
d_U(\varphi,\psi) = \inf\{\varepsilon>0 \mid&\; \forall F\subset_{fin} \lip(X)\; \exists u\in\mathcal{U}(A)\; \forall f\in F:\nonumber\\
&\|u\varphi(f)u^*-\psi(f)\| < \varepsilon\}.
\end{align}
In other words,
\begin{equation}
d_U(\varphi,\psi) = \sup_{F\subset_{fin}\lip(X)}\inf_{u\in\mathcal{U}(A)}\sup_{f\in F} \|u\varphi(f)u^*-\psi(f)\|.
\end{equation}
The \emph{Cuntz distance} is defined as 
\begin{align}
d_W(\varphi,\psi) &=\inf\{r>0 \mid \varphi(f_{U}) \precsim \psi(f_{U_r}), \: 
 \psi(f_{U}) \precsim \varphi(f_{U_r}) \: \forall \textrm{ open } U\subseteq X\}.
\end{align}
Here, for an open set $O\subseteq X$, $f_O$ is any positive continuous function whose support is $O$. Notice that the definition of $d_W(\varphi,\psi)$ does not depend on the function $f_O$ one decides to pick.

If $T(A)\neq\emptyset$, the \emph{tracial distance} is defined as
\begin{equation}
\delta(\varphi,\psi)=\sup_{\tau\in T(A)}\delta(\mu_{\varphi^*\tau},\mu_{\psi^*\tau}).
\end{equation}
\end{definition} 
Each of these three distances induces a pseudometric on the space of unital $^*$-monomorphisms $C(X)\to A$. Two morphisms have unitary distance equal to $0$ if and only if they are approximately unitary equivalent. Similarly, they have Cuntz distance zero if and only if the induced maps on the Cuntz semigroups agree, and have tracial distance equal to $0$ when they pull back the same traces.

If $X\subseteq\mathbb{C}$, unital $^*$-monomorphisms $\varphi,\psi\colon C(X)\to A$ correspond to normal elements $x,y\in A$ with spectrum $X$ (namely, $x=\varphi(\id)$ and $y=\psi(\id)$). In this case we may sometimes write $\mu_{\tau,x}$ and $\mu_{\tau,y}$ for $\mu_{\varphi^*\tau}$ and $\mu_{\psi^*\tau}$, and $d_W(x,y)$ and $\delta(x,y)$ for $d_W(\varphi,\psi)$ and $\delta(\varphi,\psi)$. Even though in general $d_U(\varphi,\psi)$ is larger than the distance
\begin{equation}
d_U(x,y) = \inf\{\|uxu^*-y\| \mid u\in\mathcal{U}(A)\}
\end{equation}
between the unitary orbits of $x$ and $y$, we will see in Theorem~\ref{main2} that under suitable hypotheses all of these distances are equal.

\begin{remark}\label{matching}
\begin{enumerate}[label=(\roman*)]
\item The tracial distance $\delta$ is an extension of the spectral distance $\delta$ of \cite{Hiai:1989aa} (which in the context of von Neumann algebras is defined via traces of spectral projections). It is denoted by $D_T$ in \cite{Hu:2015aa} and by $d_P$ in \cite{Jacelon:2014aa}. Actually, $d_P$ and $d_W$ are defined in \cite{Jacelon:2014aa} only for positive contractions and for open sets of the form $(t,1]$. In principle this means that Definition~\ref{metrics} gives larger distances, but under suitable hypotheses (for example those of \cite[Theorem 5.2]{Jacelon:2014aa}) the two versions agree.
\item For normal matrices $a$ and $b$, $\delta(a,b)$ coincides with the bottleneck distance between the two sets of eigenvalues. For self-adjoint matrices, this minimum is attained by listing both sets of eigenvalues in ascending order. For unitary matrices, the minimum can be obtained via anticlockwise orderings of the two sets of eigenvalues (see \cite{Bhatia:1984aa}).
\end{enumerate}
\end{remark}

\begin{lemma} \label{ineq}
Let $A$ be a simple, tracial, unital $\cs$-algebra.
\begin{enumerate}[label=(\roman*)]
\item \label{continuous}
There is a constant $c>0$ such that \[\delta(x,y)\le c \cdot d_U(x,y)\] for every pair of normal elements $x,y\in A$. In particular, $\delta(\cdot,\cdot)$ is uniformly continuous (as a function of approximate unitary equivalence classes of normal elements of $A$).
\item \label{commuting} If $x$ and $y$ are commuting normal elements of $A$ then \[\delta(x,y) \le d_U(x,y).\]
\item \label{dimensionfunction} If $\varphi,\psi\colon C(X)\to A$ are unital $^*$-monomorphisms then \[\delta(\varphi,\psi) \le d_W(\varphi,\psi).\]
\item \label{connected} If $X$ is connected, $A$ is exact and has strict comparison, then for unital $^*$-monomorphisms  $\varphi,\psi\colon C(X)\to A$ we have \[d_W(\varphi,\psi)= \delta(\varphi,\psi).\]
\end{enumerate}
\end{lemma}

\begin{proof}
Building on \cite{Davidson:1986aa}, (i) and (ii) were observed in \cite{Hiai:1989aa} to hold for a semifinite von Neumann algebra $M$ with a faithful normal semifinite trace $\tau$. This includes $\rm{II}_1$ factors, so in particular applies to the weak closure $M_\tau=\pi_\tau(A)''$ of the image of the GNS representation $\pi_\tau$ associated to $\tau\in T(A)$. Since
\[
d_U^A(x,y) = d_U^{\pi_\tau(A)}(\pi_\tau(x),\pi_\tau(y)) \ge d_U^{M_\tau}(\pi_\tau(x),\pi_\tau(y))
\]
and by definition
\[
\delta^A(x,y) = \sup_{\tau\in T(A)} \delta(\mu_{\tau,x},\mu_{\tau,y}) = \sup_{\tau\in T(A)} \delta^{M_\tau}(\pi_\tau(x),\pi_\tau(y))
\]
(where the superscripts indicate the algebra in which the distance should be measured), (i) follows from \cite[Theorem 1.1]{Hiai:1989aa} (which provides a universal constant $c$ such that $\delta^{M_\tau}(\pi_\tau(x),\pi_\tau(y)) \le cd_U^{M_\tau}(\pi_\tau(x),\pi_\tau(y))$) and (ii) follows from \cite[Theorem 2.1(1)]{Hiai:1989aa} (which says that $d_U^{M_\tau}(\pi_\tau(x),\pi_\tau(y)) \ge \delta^{M_\tau}(\pi_\tau(x),\pi_\tau(y))$ if $x$ and $y$ commute).

The relevant facts for (iii) are that dimension functions are order-preserving, that is, \[\varphi(f_{U}) \precsim \psi(f_{U_r}) \implies d_\tau(\varphi(f_{U})) \le d_\tau(\psi(f_{U_r}))\] and also that
\[
d_\tau(\varphi(f_{U})) = \mu_{\varphi^*\tau}(U)
\]
for any $\tau\in T(A)$.

For (iv), it remains to show that $d_W(\varphi,\psi) \le \delta(\varphi,\psi)$. The proof is similar to \cite[Lemma 2]{Cheong:2015aa}. Suppose that $r \geq \delta(\varphi, \psi)$.  Then for every tracial state $\tau \in T(A)$ and every open set $U \subset X$, we have 
\[ \mu_{\varphi^*\tau}(U) \leq \mu_{\psi^*\tau}(U_r), \qquad \mu_{\psi^*\tau}(U)) \leq \mu_{\varphi^*\tau}(U_r),\]
which in turn implies
\[ d_{\tau}(\varphi(f)) \leq d_{\tau}(\psi(f_r)), \qquad d_{\tau}(\psi(f)) \leq d_{\tau}(\varphi(f_r)),\]
where $f$ is any positive function which is nonzero on $U$ and $f_r$ is any positive function which is nonzero on $U_r$.  

Since $A$ has strict comparison and connectedness of $X$ implies that $f$ and $f_r$ are never equivalent to projections, \cite[Proposition 5.9]{PereraToms2007} implies that 
\[ \varphi(f) \precsim \psi(f_r), \qquad \psi(f) \precsim \varphi(f_r),\]
and the result follows.
\end{proof}

\begin{proposition} \label{local}
Let $X$ be a compact metric space, let $A$ be a simple, tracial, unital $\cs$-algebra. Then, for any pair of unital $^*$-monomorphisms $\varphi,\psi\colon C(X)\to A$, we have
\begin{equation} \label{pointwise}
\delta(\varphi,\psi) = \sup_{f\in\lip(X)}\delta(\varphi(f),\psi(f)).
\end{equation}
\end{proposition}

\begin{proof}
For a metric space $Y$, $O\subseteq Y$ open and $\mu,\nu\in \mathcal{M}(Y)$, write
\[
\delta_O(\mu,\nu) = \inf\{r>0 \mid \mu(O) \le \nu(O_r), \: \nu(O) \le \mu(O_r)\},
\]
so that
\[
\delta(\mu,\nu)=\sup_{O\subseteq Y}\delta_O(\mu,\nu).
\]
For every $f\in\lip(X)$, every open set $U\subseteq f(X)$ and every $r>0$ we have $f^{-1}(U)_r \subseteq f^{-1}(U_r)$. Therefore,
\begin{align*}
\delta(\varphi(f),\psi(f)) &= \sup_{\tau\in T(A)} \delta(\mu_{\tau,\varphi(f)},\mu_{\tau,\psi(f)})\\
&= \sup_{\tau\in T(A)} \delta(f_*\mu_{\varphi^*\tau},f_*\mu_{\psi^*\tau})\\
&= \sup_{\tau\in T(A)} \sup_{U\subseteq f(X)} \delta_{U}(\mu_{\varphi^*\tau}\circ f^{-1}, \mu_{\psi^*\tau}\circ f^{-1})\\
&\le \sup_{\tau\in T(A)} \sup_{U\subseteq f(X)} \delta_{f^{-1}(U)}(\mu_{\varphi^*\tau}, \mu_{\psi^*\tau})\\
&\le \sup_{\tau\in T(A)} \sup_{O \subseteq X} \delta_{O}(\mu_{\varphi^*\tau}, \mu_{\psi^*\tau})\\
&=\delta(\varphi,\psi).
\end{align*}
For the reverse inequality, observe that for every open subset $O\subseteq X$  the function $g_0 \colon X\to\mathbb{C}$ defined by $g_O(x)=d(x,\bar{O})$ is in $\lip(X)$. Therefore,
\begin{align*}
\delta(\varphi,\psi) &= \sup_{\tau\in T(A)} \sup_{O\subseteq X} \delta_{O} (\mu_{\varphi^*\tau},\mu_{\psi^*\tau})\\
&= \sup_{\tau\in T(A)} \sup_{O\subseteq X} \sup_{t>0} \delta_{O_t} (\mu_{\varphi^*\tau},\mu_{\psi^*\tau})\\
&= \sup_{\tau\in T(A)} \sup_{O\subseteq X} \sup_{t>0} \delta_{[0,t)}(\mu_{\varphi^*\tau}\circ g_O^{-1},\mu_{\varphi^*\tau}\circ g_O^{-1})\\
&\le \sup_{\tau\in T(A)} \sup_{O\subseteq X} \sup_{U\subseteq[0,\infty]} \delta_{U}(\mu_{\varphi^*\tau}\circ g_O^{-1},\mu_{\varphi^*\tau}\circ g_O^{-1})\\
&= \sup_{\tau\in T(A)} \sup_{O\subseteq X} \delta(\mu_{\tau,\varphi(g_O)},\mu_{\tau,\psi(g_O)})\\
&\le \sup_{\tau\in T(A)}  \sup_{f\in\lip(X)}\delta(\mu_{\tau,\varphi(f)},\mu_{\tau,\psi(f)})\\
&= \sup_{f\in\lip(X)}\delta(\varphi(f),\psi(f)). \qedhere
\end{align*}
\end{proof}

\begin{corollary} \label{continuitytrick}
Let $A$ be a simple, tracial, unital $\cs$-algebra and let $X$ be a compact metric space. Then $\delta(\cdot,\cdot)$ is uniformly continuous with respect to $d_U$ (as a function defined on approximate unitary equivalence classes of unital $^*$-monomorphisms $C(X)\to A$).
\end{corollary}

\begin{proof}
Let $\varphi\colon C(X)\to A$ be a unital $^*$-monomorphism and let $\varepsilon>0$. Let $c$ be as in Lemma~\ref{ineq}(\ref{continuous}). Suppose that $\varphi'\colon C(X)\to A$ is a unital $^*$-monomorphism with $d_U(\varphi,\varphi')<\frac{\varepsilon}{2c}$. Then, in particular, for every function $f\in\lip(X)$ there is a unitary $u_f\in A$ such that $\|u_f\varphi(f)u_f^*-\varphi'(f)\|<\frac{\varepsilon}{2c}$. Lemma~\ref{ineq}\ref{continuous} implies that $\delta(\varphi(f),\varphi'(f)) < \frac{\varepsilon}{2}$ for every $f\in\lip(X)$ and so by Proposition~\ref{local}, $\delta(\varphi,\varphi') < \varepsilon$.
\end{proof}

\begin{corollary} \label{commutingtrick}
Let $A$ be a simple, tracial, unital $\cs$-algebra and let $X$ be a compact metric space. Suppose that $\varphi,\psi\colon C(X)\to A$ are two unital $^*$-monomorphisms with commuting images. Then $\delta(\varphi,\psi) \le d_U(\varphi,\psi)$.
\end{corollary}

\begin{proof}
By Proposition~\ref{local} and Lemma~\ref{ineq}\ref{commuting} we have
\begin{align*}
\delta(\varphi,\psi) &= \sup_{f\in\lip(X)} \delta(\varphi(f),\psi(f))\\
&\le \sup_{f\in\lip(X)} \inf_{u\in \mathcal{U}(A)} \|u\varphi(f)u^*-\psi(f)\|\\
&= \sup_{F\subset_{fin} \lip(X)} \sup_{f\in F} \inf_{u\in\mathcal{U}(A)} \|u\varphi(f)u^*-\psi(f)\|\\
&\le \sup_{F\subset_{fin} \lip(X)} \inf_{u\in\mathcal{U}(A)} \sup_{f\in F} \|u\varphi(f)u^*-\psi(f)\|\\
&= d_U(\varphi,\psi). \qedhere
\end{align*}
\end{proof}

\section{The classification machine} \label{mainsec}

In this section we use classification machinery to convert transport maps in the commutative $\cs$-algebra $C(X)$ into conjugating unitaries in a well-behaved $\cs$-algebra $A$. To do so, we will use some (basic) $KL$-theory, where $KL$ is $K$-theory with coefficients. For a reference to the main definitions and properties see \cite[\S2.4]{Rordam.Class}. The following describes domains and targets that are amenable to this strategy.

\begin{definition}
Let $\mathcal{C}_1$ denote the class of infinite-dimensional, separable, simple, unital, exact $\cs$-algebras with real rank zero, stable rank one, weakly unperforated $K_0$ and finitely many extremal tracial states (which we will typically write as $\tau_1,\ldots,\tau_m$). Let $\mathcal{C}_2$ denote the class of separable, simple, unital, exact, $\mathcal{Z}$-stable $\cs$-algebras with a unique tracial state (which we will typically write as $\tau$).
\end{definition}

\begin{remark}
Recall that a separable, unital $\cs$-algebra $A$ is called $\mathcal{Z}$-\emph{stable} if either of the following equivalent conditions holds:
\begin{enumerate}[label=(\roman*)]
\item \label{tensor} $A\cong A\otimes\mathcal{Z}$;
\item \label{central} for any finite subsets $F\subseteq A$ and $G\subseteq\mathcal{Z}$, and any $\varepsilon>0$, there is a unital $(G,\varepsilon)$-multiplicative $^*$-linear map $\varphi\colon \mathcal{Z}\to A$ such that $\|[\varphi(b),a]\|<\varepsilon$ for every $a\in F$ and $b\in G$.
\end{enumerate}
(See \cite[Theorem 7.2.2]{Rordam.Class} and \cite[Theorem 2.2]{Toms:2007uq}.) Outside of the separable setting, as in Theorem~\ref{nonseparable}, the right notion of $\mathcal{Z}$-stability is \ref{central} (see \cite[Theorem 4.1]{Farah:2014aa}).
\end{remark}

\begin{definition}\label{def:Kcontr}
Let $X$ be a compact connected metric space. Recall that $X$ is \emph{$K$-contractible} if it has the same $K$-theory as a point. Let us say that $X$ is \emph{$K$-planar} if it has the $K$-theory of a subset of the plane, so $K^0(X)\cong\mathbb{Z}$ and $K^1(X)$ is free and abelian---see \cite[Proposition 7.5.2]{Higson:2000to}.
\end{definition}

\begin{remark}\label{free}
Typically, uniqueness results for unital $^*$-monomorphisms $\varphi,\psi\colon C(X)\to A$ are of the following form: $\varphi$ and $\psi$ are approximately unitarily equivalent if and only if 
\begin{enumerate}[label=(\roman*)]
\item $KL(\varphi)=KL(\psi)$,
\item $\varphi$ and $\psi$ agree on algebraic $K_1$, and
\item $\tau\circ\varphi = \tau\circ\psi$ for every tracial state $\tau\in T(A)$.
\end{enumerate}
The assumptions we make on $X$, $A$, $\varphi$ and $\psi$ ensure that agreement on the $K$-theoretic part of the invariant (detailed descriptions of which can be found in \cite[\S2 and \S3]{Matui:2011uq}) holds automatically. More precisely:
\begin{items}
\item if $X$ is $K$-planar then $KL(\varphi)=KL(\psi)$ if and only if  $K_*(\varphi)=K_*(\psi)$, which is immediate if $\varphi$ and $\psi$ are unital and $K_1$-trivial;
\item algebraic $K_1$ becomes redundant if $A$ has real rank zero or $K^1(X)=0$.
\end{items}
In the sequel, this allows us to isolate and fine-tune the measure-theoretic part of the invariant.
\end{remark}

First, we adapt some existence and uniqueness results from the literature to suit our needs. We have found it convenient to use \cite{Matui:2011uq} as a reference, but this is an ongoing research area and there are many other relevant articles, for example \cite{Gong:2000ys}.

\begin{proposition} \label{existence}
Let $X$ be a compact connected $K$-contractible metric space and let $\mu$ be a fully supported Borel probability measure on $X$. Then, for every $A\in\mathcal{C}_2$, there is a unital $^*$-monomorphism $\varphi\colon C(X)\to A$ with $\varphi^*\tau=\tau_\mu$.
\end{proposition}

\begin{proof}
If $X$ is locally connected, the argument is relatively easy. By \cite[Theorem 2.1]{Rordam:2004kq} there is a unital $^*$-monomorphism
\[
\varphi'\colon C([0,1]) \hookrightarrow 1\otimes\mathcal{Z} \hookrightarrow A\otimes\mathcal{Z} \cong A
\]
such that $\tau\circ\varphi'=\tau_\lambda$, where $\lambda$ is Lebesgue measure on $[0,1]$. By \cite[Theorem 1]{Kolesnikov:1999aa} there is a continuous surjection $\pi\colon [0,1]\to X$ with $\pi_*\lambda=\mu$. Then $\varphi=\varphi'\circ\pi^*\colon C(X)\to A$ is the required map.

For the general case we argue as follows. It is well known that the compact connected metric space $X$ may be written as an inverse limit $X=\varprojlim (X_i,f_i)$, where each $X_i$ is a finite simplicial complex and each continuous function $f_i\colon X_{i+1}\to X_i$ is surjective (see \cite[Corollary 4.10.11]{Sakai:2013aa}). Dually, $C(X)\cong \varinjlim (C(X_i),\theta_i)$, where $\theta_i=f_i^*$. Since by assumption $K_0(C(X))=\langle[1]\rangle=\mathbb{Z}$ and $K_1(C(X))=0$, we may assume by passing to a subsequence that, for every $i$, $K_1(\theta_i)=0$ and there are generating sets $S_i=\{[1]=s_1^i,\ldots,s_{n_i}^i\}\subseteq K_0(C(X_i))$ such that $K_0(\theta_i)(s_j^i)=0$ for $j>1$.

For each $i$, the pushforward $\mu_i$ of $\mu$ to $X_i$ is a fully supported Borel probability measure. Since each $X_i$ is a Peano continuum, by the first part of the proof there are unital $^*$-monomorphisms $\varphi_i\colon C(X_i)\to A$ with $\varphi_i^*\tau=\tau_{\mu_i}$. Then the maps $\Phi_i=\varphi_{i+2}\circ\theta_{i+1}\circ\theta_i$ and $\Psi_i=\varphi_{i+1}\circ\theta_i$ are such that
\begin{enumerate}[label=(\roman*)]
\item $KL(\Phi_i)=KL(\Psi_i)$ (or equivalently in this setting, $K_*(\Phi_i)=K_*(\Psi_i)$),
\item $\Phi_i$ and $\Psi_i$ agree on algebraic $K_1$ (because by construction they both factor through $C([0,1])$, which has trivial $K_1$), and
\item $\tau\circ\Phi_i = \tau\circ\Psi_i \:(=\tau_{\mu_i})$.
\end{enumerate}
By \cite[Theorem 6.6]{Matui:2011uq}, we therefore have that $\Phi_i$ and $\Psi_i$ are approximately unitarily equivalent. It follows that the diagram
\[
\begin{tikzcd}
\cdots \arrow[r] & C(X_i) \arrow[r,"\theta_i"] \arrow[d,"\varphi_i"] & C(X_{i+1}) \arrow[r,"\theta_{i+1}"] \arrow[d,"\varphi_{i+1}"] & C(X_{i+2})\arrow[r] \arrow[d,"\varphi_{i+2}"] & \cdots \arrow[r] & C(X) \arrow[d,dashed,"\varphi"]\\
\cdots \arrow[r,equal] & A \arrow[r,equal] & A \arrow[r,equal] & A \arrow[r,equal] & \cdots \arrow[r,equal] & A
\end{tikzcd}
\]
is an approximate intertwining inducing the required unital $^*$-monomorphism $\varphi\colon C(X)\to A$.
\end{proof}

\begin{proposition} \label{uniqueness}
Suppose that either $A\in\mathcal{C}_1$ and $X$ is a compact connected metric space or $A\in\mathcal{C}_2$ and $X$ is a compact connected metric space which is moreover $K$-contractible. Let  $\varphi\colon C(X) \to A$ be a unital \mbox{$^*$-monomorphism}. Then for every $\varepsilon>0$ there exists $\delta>0$ and a finite set of positive contractions $G\subseteq C(X)$ such that the following holds. If $\varphi'\colon C(X)\to A$ is a unital $^*$-monomorphism with $KL(\varphi')=KL(\varphi)$ and $|\tau(\varphi'(g))-\tau(\varphi(g))|<\delta$ for every $g\in G$ and $\tau\in T(A)$, then $d_U(\varphi,\varphi')<\varepsilon.$ That is, for any finite subset $F\subseteq\lip(X)$ there is a unitary $u\in A$ such that \[\|\varphi'(f)-u\varphi(f)u^*\| < \varepsilon\] for every $f\in F$.
\end{proposition}

\begin{proof}
This follows from \cite[Theorems 4.7 and 6.6]{Matui:2011uq} since $\lip(X)$ is compact. (Alternatively, one can directly check that in the proofs of the results of \cite{Matui:2011uq} the set $G$ of test functions and the tolerance $\delta>0$ depend only on the modulus of continuity of the elements of $F$---see \cite[Lemmas 4.1--4.3]{Matui:2011uq}.)
\end{proof}

For measures with atoms transport maps may not even exist, optimal or otherwise. In our context we do not lose any generality by excluding these.

\begin{proposition} \label{diffuse}
Suppose that either $A\in\mathcal{C}_1$ and $X$ is a compact path-connected metric space or $A\in\mathcal{C}_2$ and $X$ is a compact connected metric space which is moreover $K$-contractible. Let  $\varphi\colon C(X) \to A$ be a unital \mbox{$^*$-monomorphism} and $\varepsilon>0$. Then, there exists a unital $^*$-monomorphism $\varphi'\colon C(X)\to A$ such that
\begin{enumerate}[label=(\roman*)]
\item $KL(\varphi')=KL(\varphi)$,
\item $d_U(\varphi,\varphi') < \varepsilon$, and
\item $\mu_{(\varphi')^*\tau}\in\mathcal{M}_g(X)$ for every $\tau\in T(A)$.
\end{enumerate}
\end{proposition}

\begin{proof}
In either case, let $\tau_1,\ldots,\tau_m$ be the extremal tracial states of $A$. Let $G\subseteq C(X)$ and $\delta>0$ be as in Proposition~\ref{uniqueness}. Since $\mathcal{M}_g(X)$ is weak-$^*$-dense in the set of faithful Borel probability measures on $X$, there exist $\mu_1,\ldots,\mu_m\in\mathcal{M}_g(X)$ such that $|\varphi^*\tau_i(g)-\tau_{\mu_i}(g)|<\delta$ for $i=1,\ldots,m$ and $g\in G$.

By \cite[Theorem 2.6]{Matui:2011uq} (which is a simplification of \cite[Theorem 0.1]{Ng:2008aa}) or Proposition~\ref{existence}, there is a unital $^*$-monomorphism $\varphi'\colon C(X)\to A$ with $KL(\varphi')=KL(\varphi)$ and $(\varphi')^*\tau_i=\tau_{\mu_i}$ for every $i$. By convexity, $\varphi'$ is the required map.
\end{proof}

In the real rank zero setting, we can simultaneously transport finitely many measures by diagonalising. In what follows, $\delta_{ij}$ is Kronecker's function.

\begin{lemma} \label{directsum}
Let $A\in\mathcal{C}_1$ with extremal traces $\tau_1,\ldots,\tau_m$. Then for every $\gamma>0$, there exist orthogonal projections $p_1,\ldots,p_m\in A$ such that
\begin{enumerate}[label=(\roman*)]
\item $p_1+\cdots+p_m=1$ and
\item $|\tau_j(p_i) - \delta_{ij}| < \gamma$ for $1\le i,j\le m$.
\end{enumerate}
\end{lemma}

\begin{proof}
This follows from the results of \cite[\rm{III}]{Blackadar:1982kq} (summarised for example in \cite[Theorems 6.9.1, 6.9.2 and 6.9.3]{Blackadar:1998qf}). Although it will be familiar to experts, we include the details for completeness.

Choose $\varepsilon>0$ such that $\frac{\varepsilon}{1-\varepsilon}+\varepsilon < \gamma$. Let us write $\mathcal{C}_1(m)$ for those elements of $\mathcal{C}_1$ with $m$ extremal traces. For $1\le i\le m$, let $f_i\in\aff(T(A))$ be the function defined by $f_i(\tau_j)=\delta_{ij}$. Let $f_1'\in\aff(T(A))$ be a function whose range is contained in $(\frac{\varepsilon}{2},1-\frac{\varepsilon}{2})$ such that \[\max_{\tau\in T(A)}|f_1'(\tau)-f_1(\tau)|<\frac{\varepsilon}{2}.\] By \cite[Theorem 6.9.3]{Blackadar:1998qf}, there exists $x=[q]-[p]\in K_0(A)$ such that $|\tau_*(x)-f_1'(\tau)|<\frac{\varepsilon}{2}$ for every $\tau\in T(A)$. Here, $p$ and $q$ are projections in some matrix algebra over $A$ and $\tau_*(x):=\tau(q)-\tau(p)$. By choice of $f_1'$, we have $\{\tau_*(x) \mid \tau\in T(A)\} \subseteq (0,1)$. In particular, $\tau(p)<\tau(q)$ for every $\tau\in T(A)$ so by \cite[Corollary 6.9.2]{Blackadar:1998qf} there is a projection $p_1$ over $A$ such that $x=[q]-[p]=[p_1]$. Moreover, since $\tau(p_1)=\tau_*(x)<1=\tau(1)$ for every $\tau\in T(A)$, we may assume $p_1\in A$. Finally, note that

\begin{equation} \label{tracep1}
\max_{1\le j\le m}|\tau_j(p_1) - \delta_{1j}| = \max_{\tau\in T(A)}|\tau(p_1) - f_1(\tau)| < \varepsilon.
\end{equation}

Now we proceed by induction, using the fact that since $A$ is in $\mathcal{C}_1(m)$ then so is $B:=(1-p_1)A(1-p_1)$. In particular, by \cite[Lemma 4.6, Proposition 4.7]{Cuntz:1979fv} the inclusion map $B\hookrightarrow A$ induces an isomorphism of tracial cones; equivalently, every tracial state $\tau\in T(B)$ is the restriction of a unique trace $\tau\in\widetilde{T(A)}=\{\lambda\cdot \sigma \mid \sigma\in T(A), \lambda\in[0,\infty)\}$ with $\tau(1-p_1)=1$.

By induction, there exist projections $p_2,\ldots,p_m\in (1-p_1)A(1-p_1)$ with $p_2+\cdots+p_m=1-p_1$ such that for $2\le i,j\le m$,
\begin{equation} \label{tracep2}
\left|\frac{\tau_j(p_i)}{\tau_j(1-p_1)}-\delta_{ij}\right|<\varepsilon.
\end{equation}
The case $i=1$ of the lemma is covered by (\ref{tracep1}), as is the case $j=1$: for $i>1$, 
\[
\tau_1(p_i) \le \tau_1(1-p_1) = 1-\tau_1(p_1) < \varepsilon <\gamma.
\]
For $2\le i,j\le m$, we have $\tau_j(1-p_1)=1-\tau_j(p_1) > 1-\varepsilon$, hence by (\ref{tracep2}),
\begin{align*}
|\tau_j(p_i)-\delta_{ij}| &< \left|\tau_j(p_i)-\frac{\tau_j(p_i)}{\tau_j(1-p_1)}\right|+\varepsilon\\
&\le \frac{\tau_j(p_1)}{1-\tau_j(p_1)}\|p_i\|+\varepsilon\\
&< \frac{\varepsilon}{1-\varepsilon}+\varepsilon\\
&< \gamma. \qedhere
\end{align*}
\end{proof}

\begin{proposition} \label{diagonal}
Let $A\in\mathcal{C}_1$. Let $X$ be a compact path-connected $K$-planar metric space. Let $\psi\colon C(X)\to A$ be a unital $^*$-monomorphism. Then for every $\varepsilon>0$, there are orthogonal projections $p_1,\ldots,p_m\in A$ with $p_1+\cdots+p_m=1$ and unital $^*$-monomorphisms $\psi_i\colon C(X)\to p_iAp_i$ such that the unital $^*$-monomorphism $\psi'\colon C(X)\to A$ given by the factorisation
\[ \label{commute}
\begin{tikzcd}
C(X) \arrow[rr,dashed,"\psi' "] \arrow[dr,"\bigoplus_{i=1}^m\psi_i"] & & A\\
& \bigoplus_{i=1}^m p_iAp_i \arrow[ur, hook] & 
\end{tikzcd}
\]
satisfies
\begin{enumerate}[label=(\roman*)]
\item $KL(\psi')=KL(\psi)$,
\item $d_U(\psi,\psi') < \varepsilon$, and
\item $|\psi_i^*\tau_i(g)-(\psi')^*\tau_i(g)|<\varepsilon$ for every $i$ and every positive contraction $g\in C(X)$.
\end{enumerate}
\end{proposition}

\begin{proof}
Let $G\subseteq C(X)$ and $\delta\in(0,\varepsilon)$ be as required in Proposition~\ref{uniqueness} for $\varepsilon$ and $\psi$. Choose $\gamma>0$ such that $\gamma + \frac{\gamma}{1-\gamma} < \delta$ and let $p_1,\ldots,p_m\in A$ be as in Lemma~\ref{directsum} for this $\gamma$ and the functions $\delta_{ij}$. Then by \cite[Theorem 2.6]{Matui:2011uq} there are unital $^*$-monomorphisms $\psi_i\colon C(X)\to p_iAp_i$ with:
\begin{enumerate}[label=(\roman*)] 
\item $\psi_i^*\tau = \psi^*\tau_i$ for every $\tau\in T(p_iAp_i)$,
\item $K_1(\psi_1) = K_1(\psi)$ (under the isomorphism $K_1(p_1Ap_1)\cong K_1(A)$ induced by inclusion),
\item $K_1(\psi_i) = 0$ for $i>1$.
\end{enumerate}
Then the $^*$-homomorphism $\psi' = \sum_{i=1}^m \psi_i$ is injective, unital and satisfies $K_1(\psi') = K_1(\psi)$ (by construction) and $K_0(\psi')=K_0(\psi)$ (since both maps are unital). (See Remark~\ref{free}.) This means that $KL(\psi')=KL(\psi)$. Moreover, $\psi'$ and $\psi$ approximately agree on traces: for $i=1,\ldots,m$ and $g\in G$ we have
\begin{align*}
|\tau_i(\psi'(g)) - \tau_i(\psi(g))| & = \left|\tau_i\left(\sum_{j=1}^m\psi_j(g)\right) - \tau_i(\psi(g))\right|\\
&\le \sum_{j\ne i}\tau_i(\psi_j(g))  + |\tau_i(\psi_i(g)) - \tau_i(\psi(g))|\\
&\le \tau_i(1-p_i) + \left|\tau_i(\psi_i(g)) - \frac{\tau_i}{\tau_i(p_i)}(\psi_i(g))\right|\\
&+ \left|\frac{\tau_i}{\tau_i(p_i)}(\psi_i(g)) - \tau_i(\psi(g))\right|\\
&\le \gamma + \left|\frac{1}{\tau_i(p_i)}-1\right| +0\\
&\le \gamma + \frac{\gamma}{1-\gamma}\\
&< \delta.
\end{align*}
The conclusion then follows by choice of $G$ and $\delta$.
\end{proof}

In our first theorem we consider normal elements in $\cs$-algebras with real rank zero that have Peano continuua as spectra.

\begin{theorem} \label{main1}
Let $A\in\mathcal{C}_1$ and let $X\subseteq\mathbb{C}$ be a Peano continuum. Suppose that $x, y \in A$ are normal elements  corresponding to unital $^*$-monomorphisms $\varphi,\psi\colon C(X)\to A$ with $K_1(\varphi)=K_1(\psi)=0$. Then \[d_U(x,y) = \delta(x,y) = d_W(x,y).\]
\end{theorem}

\begin{proof}
Since $X$ is connected and the hypotheses imply that $A$ has strict comparison of positive elements (see \cite[Corollary 3.10]{Perera:1997zl}), Lemma~\ref{ineq}\ref{connected} gives $\delta(x,y) = d_W(x,y)$. The inequality $d_U(x,y) \le \delta (x,y)$ is then provided by \cite[Theorem 3.6]{Hu:2015aa} (which is proved by finite-spectra approximations).

We show that $\delta(x,y) \le d_U(x,y)$ as follows. By Proposition~\ref{diffuse}, we may assume that the measures $\mu_i:=\mu_{\varphi^*\tau_i}$ and $\nu_i:=\mu_{\psi^*\tau_i}$ are in $\mathcal{M}_g(X)$. Since the spectrum $X$ is a Peano continuum, for each $i\in\{1,\ldots,m\}$ there exists a continuous surjection $h_i\colon X\to X$ such that $(h_i)_*\nu_i=\mu_i$ (see the discussion in Section~\ref{measures}). Let $\varepsilon>0$. Let $G\subseteq C(X)$ and $\delta>0$ be as required in Proposition~\ref{uniqueness} for $\frac{\varepsilon}{2c}$ and $\varphi$ (where $c$ is as in Lemma~\ref{ineq}\ref{continuous}). By Proposition~\ref{diagonal}, we may assume that $\psi$ factorises as
\[
\bigoplus_{i=1}^m\psi_i\colon  C(X) \to \bigoplus_{i=1}^m p_iAp_i \hookrightarrow A
\]
with $|\psi_i^*\tau_i(g)-\psi^*\tau_i(g)|<\frac{\delta}{2}$ for every $i$ and every positive contraction $g\in  C(X)$. Now define the unital $^*$-monomorphism $\varphi'\colon C(X)\to A$ by $\varphi'=\bigoplus_{i=1}^m(\psi_i\circ h_i^*)$. Then $KL(\varphi')=KL(\psi)=KL(\varphi)$ and $|\tau(\varphi'(g)) - \tau(\varphi(g))|<\delta$ for every $\tau\in T(A)$, so by Proposition~\ref{uniqueness}, $d_U(\varphi',\varphi)<\frac{\varepsilon}{2c}$.

Let $x'=\varphi'(\id)$. Then $x'$ commutes with $y$ and by the above conclusion there is a unitary $u\in\mathcal{U}(A)$ such that $\|x'-uxu^*\|<\frac{\varepsilon}{2c}$. By \ref{continuous} and \ref{commuting} of Lemma~\ref{ineq} we therefore have that
\[
\delta(x,y) \le \delta(x',y) + \frac{\varepsilon}{2} \le d_U(x',y) + \frac{\varepsilon}{2} \le d_U(x,y) + \varepsilon.\qedhere
\]
\end{proof}

If we are willing to demand more of the compact metric space $X$ (specifically, continuous transport) then we can refine the inequality $d_U(x,y) \le \delta(x,y)$ to $d_U(\varphi,\psi) \le \delta(\varphi,\psi)$ and in some cases even remove the assumption of real rank zero.

\begin{theorem} \label{main2}
Let $A\in\mathcal{C}_1\cup\mathcal{C}_2$, let $X$ be a compact path-connected $K$-planar metric space that admits approximate continuous transport and let $\varphi,\psi\colon C(X)\to A$ be unital $^*$-monomorphisms. Suppose that either
\begin{enumerate}[label=(\roman*)]
\item \label{rr0} $A\in\mathcal{C}_1$ and $K_1(\varphi) = K_1(\psi) = 0$ or
\item \label{zstable} $A\in\mathcal{C}_2$ and $X$ is $K$-contractible.
\end{enumerate}
Then
\[
d_U(\varphi,\psi) = \delta(\varphi,\psi) = d_W(\varphi,\psi).
\]
\end{theorem}

\begin{proof}
The argument is the same as in Theorem~\ref{main1}, except that here both inequalities $d_U\le\delta$ and $\delta\le d_U$ will be witnessed by the transport maps.

As before, by connectedness of $X$ and strict comparison in $A$ (which in case \ref{zstable} is provided by \cite[Corollary 4.6]{Rordam:2004kq}) we automatically have $d_W(\varphi,\psi)= \delta(\varphi,\psi)$. We also assume that the measures $\mu_i:=\mu_{\varphi^*\tau_i}$ and $\nu_i:=\mu_{\psi^*\tau_i}$ are in $\mathcal{M}_g(X)$.

Let $\varepsilon>0$ and let $F$ be a finite subset of $\lip(X)$. Choose (by Corollary~\ref{continuitytrick}) $\gamma\in(0,\frac{\varepsilon}{2})$ such that $\delta(\varphi',\varphi) <\frac{\varepsilon}{2}$ whenever $d_U(\varphi',\varphi)<\gamma$ (in fact, $\gamma=\frac{\varepsilon}{2c}$ will do). Let $G\subseteq C(X)$ and $\delta>0$ be as required in Proposition~\ref{uniqueness} for $\gamma$ and $\varphi$. By assumption there exist homeomorphisms $h_1,\ldots,h_m\colon X\to X$ such that
\begin{equation}\label{equation:transport}
\sup_{x\in X}d(h_i(x),x) \le \delta(\mu_i,\nu_i)+\frac{\varepsilon}{2}
\end{equation}
and
\[
\left|\int_Xgd(h_i)_*\nu_i-\int_Xgd\mu_i\right| < \frac{\delta}{3}
\]
for every $g\in G$ and $i=1,\ldots,m$. By Proposition~\ref{diagonal} we may assume that $\psi$ factorises as
\[
\bigoplus_{i=1}^m\psi_i\colon  C(X) \to \bigoplus_{i=1}^m p_iAp_i \hookrightarrow A
\]
with $|\psi_i^*\tau_i(g)-\psi^*\tau_i(g)|<\frac{\delta}{3}$ for every $i$ and every positive contraction $g\in C(X)$. Now define the unital $^*$-monomorphism $\varphi'\colon C(X)\to A$ by $\varphi'=\bigoplus_{i=1}^m(\psi_i\circ h_i^*)$. Then $KL(\varphi')=KL(\psi)=KL(\varphi)$ and $\varphi'$ and $\varphi$ tracially agree within $\delta$ on each element of $G$: for every $i$ and $g\in G$ we have
\begin{align*}
|\tau_i(\varphi(g)) - \tau_i(\varphi'(g))| &= \left|\int_Xgd\mu_i - \sum_{j=1}^m\int_X(h_j)^*gd\mu_{\psi_j^*\tau_i}\right| \\
&\le \left|\int_Xgd\mu_i - \int_X(h_i)^*gd\mu_{\psi_i^*\tau_i}\right| + \left|\sum_{j\ne i}\int_X(h_j)^*gd\mu_{\psi_j^*\tau_i}\right|\\
&\le \left|\int_Xgd\mu_i - \int_X(h_i)^*gd\mu_{\psi^*\tau_i}\right| + \frac{\delta}{3} + \tau_i(1-p_i)\\
&<\frac{\delta}{3}+\frac{\delta}{3}+\frac{\delta}{3}\\
&=\delta.
\end{align*}
Therefore $d_U(\varphi',\varphi)<\gamma<\frac{\varepsilon}{2}.$ 
Note moreover that $\varphi'$ and $\psi$ have commuting images. Then:
\begin{align*}
d_U(\varphi,\psi) &\le d_U(\varphi',\psi) + \frac{\varepsilon}{2}\\
&=\sup_{F\subset_{fin}\lip(X)}\inf_{u\in\mathcal{U}(A)}\sup_{f\in F} \|u\varphi'(f)u^*-\psi(f)\| + \frac{\varepsilon}{2}\\
&\le \sup_{f\in\lip(X)}\|\varphi'(f)-\psi(f)\| + \frac{\varepsilon}{2}\\
&= \sup_{f\in\lip(X)}\left\|\bigoplus_{i=1}^m(\psi_i\circ h_i^*)(f) - \bigoplus_{i=1}^m\psi_i(f)\right\| + \frac{\varepsilon}{2}\\
&= \sup_{f\in\lip(X)}\max_{1\le i\le m} \|f\circ h_i-f\| + \frac{\varepsilon}{2}\\
&\le \max_{1\le i\le m} \delta(\mu_i,\nu_i) + \varepsilon \qquad \textrm{(by \eqref{equation:transport} and since } f \textrm{ is Lipschitz})\\
&= \delta(\varphi,\psi) + \varepsilon\\
&\le \delta(\varphi',\psi) + \frac{3\varepsilon}{2}\\
&\le d_U(\varphi',\psi) + \frac{3\varepsilon}{2} \qquad \textrm{(by Corollary~\ref{commutingtrick})}\\
&\le d_U(\varphi,\psi) + 2\varepsilon.
\end{align*}
It follows that $d_U(\varphi,\psi) = \delta(\varphi,\psi)$.
\end{proof}

\begin{corollary} \label{unitaries}
Let $A$, $X$, $\varphi$ and $\psi$ be as in Theorem~\ref{main2}. Suppose further that $X\subseteq\mathbb{C}$, so that $\varphi$ and $\psi$ correspond to normal elements $x,y\in A$. Then
\[
d_U(x,y)=\delta(x,y)=d_W(x,y).
\]
\end{corollary}
(This in particular applies to certain self-adjoint and unitary elements with connected spectra.)

\begin{remark} \label{spooky}
The starting point of this article was Weyl's problem, and our goal has been to identify connected spectra that permit an exact solution. At the cost of assuming finitely many extremal traces, Theorem~\ref{main2} extends the one-dimensional examples of \cite[Remark
8.6]{Hu:2015aa} to cover all connected, locally connected spectra. However, the theory of continuous transport that we have developed provides a framework for moving beyond the plane and perhaps even beyond commutative domains. It also gives a clue about the right sort of generalised result: if not an exact equality of the distances $d_U$ and $\delta$ then an inequality $d_U \le c_X\cdot \delta$ for a constant $c_X$  computed in terms of the geometry of the underlying space (along the lines of Remark~\ref{polygons}).
\end{remark}

\begin{remark} \label{nonnuclear}
Among simple, separable, unital, exact $\cs$-algebras of real rank zero and with finitely many extremal tracial states, Theorem~\ref{main1} and Theorem~\ref{main2}\ref{rr0} apply, for example, to the $\mathcal{Z}$-stable ones. (Stable rank one and weakly unperforated $K_0$ are automatic for these algebras.) These include for example AF-algebras and irrational rotation algebras.

In addition, some non-$\mathcal{Z}$-stable $\cs$-algebras are included as well. For example, the reduced group $\cs$-algebra $A$ of the group $G=\ast_{n=2}^\infty \mathbb{Z}/n\mathbb{Z}$ satisfies all the required properties---see \cite[Example 2.4]{Dykema:2000aa}. Note that $A$ is not $\mathcal{Z}$-stable for the same reason that it is not approximately divisible: the associated group von Neumann algebra does not have property $\Gamma$. Note also that, since $G$ is not amenable, $A$ is nonnuclear. Moreover, by Rosenberg's Theorem \cite[Theorem A1]{Hadwin:1987aa} $A$ is not quasidiagonal, so by \cite[Theorem 3.4]{Lin:2001ab} $A$ does not have tracial rank zero. This $\cs$-algebra therefore falls outside the remit of \cite[Theorem 8.5, Remark 8.6]{Hu:2015aa}.
\end{remark}

We end the paper by showing that the assumption of separability can be dropped in part of Corollary~\ref{unitaries}. In particular we have the following extension, which applies for example to the nonseparable algebras considered in \cite{Farah:2015aa}.

\begin{theorem} \label{nonseparable}
Let $A$ be a simple, unital, exact, $\mathcal{Z}$-stable $\cs$-algebra with a unique tracial state $\tau$. Let $X$ be a compact connected $K$-contractible metric space that admits approximate continuous transport. Then, for unital $^*$-monomorphisms $\varphi,\psi\colon C(X)\to A$, we have
\[
d_U(\varphi,\psi)=\delta(\varphi,\psi)=d_W(\varphi,\psi).
\]
Suppose further that $X\subseteq\mathbb{C}$, so that $\varphi$ and $\psi$ correspond to normal elements $x,y\in A$. Then
\[
d_U(x,y)=\delta(x,y)=d_W(x,y).
\]
\end{theorem}
\begin{proof}
We will require a small amount of continuous model theory applied to $\cs$-algebras. For a reference, see~\cite{ModelTheory}.

First, notice that $d^A_U(\varphi,\psi)\leq d^B_U(\varphi,\psi)$ and similarly $d_W^A(\varphi,\psi)\leq d^B_W(\varphi,\psi)$ whenever $B\subseteq A$ is a unital $\cs$-subalgebra containing the images of both $\varphi$ and $\psi$, where $d_U^A$ denotes the distance $d_U$ when calculated in $A$. (This is because $B$ has fewer elements, and in particular unitaries, than $A$). Moreover, if $B$ is monotracial, we have that $\delta^A(\varphi,\psi)=\delta^B(\varphi,\psi)$, since the unique trace on $B$ coincides with the restriction of the unique trace $\tau$ on $A$ to $B$, and the value $\delta$ only depends on the measures $\mu_{\varphi^*\tau},\mu_{\psi^*\tau}\in\mathcal{M}(X)$. On the other hand, since $X$ is metrisable and in particular second countable, and $\lip(X)$ is separable, given $B$ as above one can always find, by adjoining the appropriate elements and working with a base for $X$, a separable unital $\cs$-algebra $C$ with $B\subseteq C\subseteq A$ and such that 
 \[
 d^C_U(\varphi,\psi)=d^A_U(\varphi,\psi)\text{ and }d^C_W(\varphi,\psi)=d^A_W(\varphi,\psi).
 \]
Fix such a $C$. By using an elementary submodel argument, (see~\cite[Theorem 2.6.2]{ModelTheory}), since the class of unital, unique-trace $\mathcal Z$-stable algebras is axiomatisable (\cite[Theorem 3.5.5]{ModelTheory}), we can find a separable simple monotracial $\mathcal Z$-stable $D$ such that $C\subseteq D\subseteq A$. (This can also be done  using Blackadar's methods for finding separable simple subalgebras inside nonseparable simple algebras, see \cite[Proposition 2.2]{Blackadar1978}.) Applying Theorem~\ref{main2} to $X$ and $D$, we have that
 \[
d^D_U(\varphi,\psi)=\delta^D(\varphi,\psi)=d^D_W(\varphi,\psi).
 \]
On the other hand, by the above, we have that $d^A_U(\varphi,\psi)=d^D_U(\varphi,\psi)$, $\delta^A(\varphi,\psi)=\delta^D(\varphi,\psi)$ and $d^A_W(\varphi,\psi)=d^D_W(\varphi,\psi)$, and therefore the thesis.
\end{proof}


\end{document}